\documentclass[12pt,a4paper,leqno]{article}
\usepackage[utf8]{inputenc}
\usepackage[T1]{fontenc}
\usepackage{lmodern}
\usepackage{xspace}
\usepackage[english,francais]{babel}
\usepackage{mathRit}
\usepackage{amssymb}
\usepackage{amsmath}
\usepackage{amsthm}
\usepackage[all]{xy}

\newcommand{\be}{\begin{otherlanguage}{english}}
\newcommand{\ee}{\end{otherlanguage}}

\theoremstyle{definition}
\newtheorem{defn}[subsection]{Définition}
\newtheorem{remq}[subsection]{Remarque}
\newtheorem{exem}[subsection]{Exemple}
\theoremstyle{plain}
\newtheorem{lemm}[subsection]{Lemme}
\newtheorem{prop}[subsection]{Proposition}
\newtheorem{theo}[subsection]{Théorème}
\newtheorem{coro}[subsection]{Corollaire}

\numberwithin{equation}{subsection}

\newcommand{\beq}{\begin{equation}}
\newcommand{\eeq}{\end{equation}}
\newenvironment{ea}{\csname eqnarray*\endcsname}{\csname endeqnarray*\endcsname}

\DeclareSymbolFont{eu}{U}{eur}{m}{n}
\SetSymbolFont{eu}{bold}{U}{eur}{b}{n}
\DeclareMathSymbol{\alpha}{\mathord}{eu}{"0B}
\DeclareMathSymbol{\beta}{\mathord}{eu}{"0C}
\DeclareMathSymbol{\delta}{\mathord}{eu}{"0E}
\DeclareMathSymbol{\epsilon}{\mathord}{eu}{"0F}
\DeclareMathSymbol{\zeta}{\mathord}{eu}{"10}
\DeclareMathSymbol{\iota}{\mathord}{eu}{"13}
\DeclareMathSymbol{\nu}{\mathord}{eu}{"17}
\DeclareMathSymbol{\pi}{\mathord}{eu}{"19}
\DeclareMathSymbol{\rho}{\mathord}{eu}{"1A}
\DeclareMathSymbol{\tau}{\mathord}{eu}{"1C}
\DeclareMathSymbol{\chi}{\mathord}{eu}{"1F}
\DeclareMathSymbol{\psi}{\mathord}{eu}{"20}

\newcommand{\N}{\mathbb{N}}
\newcommand{\Z}{\mathbb{Z}}
\newcommand{\F}{\mathbb{F}}
\newcommand{\Q}{\mathbb{Q}}
\newcommand{\cF}{\mathcal{F}}
\newcommand{\cG}{\mathcal{G}}
\newcommand{\fH}{\mathcal{H}}
\newcommand{\fL}{\mathcal{L}}
\newcommand{\fm}{\mathfrak{m}}
\newcommand{\al}{\alpha}
\newcommand{\eps}{\epsilon}
\newcommand{\bsl}{\backslash}

\newcommand{\Lt}{\otimes^L}
\newcommand{\Qlb}{\overline{\Q_l}}
\newcommand{\Qlalb}{\overline{\Q_{l_{\al}}}}
\newcommand{\Dbc}{D^b_c}
\newcommand{\RcHom}{R\mathcal{H}\mathit{om}}
\newcommand{\xb}{{\bar{x}}}
\newcommand{\yb}{{\bar{y}}}
\newcommand{\bb}{{\bar{b}}}
\newcommand{\kb}{{\bar{k}}}
\newcommand{\Xkb}{X_\kb}
\newcommand{\lio}{{(l,\iota)}}
\newcommand{\lion}{{\lio,n}}
\newcommand{\liow}{{\lio,w}}
\newcommand{\lidI}{{\lio\in I}}
\newcommand{\aldA}{{\al \in A}}
\newcommand{\mt}{&\mapsto&}
\newcommand{\ndP}{{n\in P}}
\newcommand{\Rl}{\mathrm{Rl}}
\newcommand{\sto}[1][]{
\xrightarrow{#1}} 
\newcommand{\har}[1][]{\stackrel{#1}{\hookrightarrow}}
\newcommand{\adhx}{\overline{\{x\}}}

\newcommand{\pH}{{}^p H}
\newcommand{\PH}{{}^P H}
\newcommand{\tle}[2][]{\tau^{#1}_{\le#2}}

\newcommand{\tYle}[2][]{\tle[Y_{#1}]{#2}}
\newcommand{\ptle}[2][]{{}^p \tau^{#1}_{\le#2}}
\newcommand{\ptpg}[1]{{}^p \tau_{>#1}}
\newcommand{\ptYle}[1]{\ptle[Y]{#1}}
\newcommand{\Ptle}[1]{{}^P \tau_{\le #1}}
\newcommand{\Ptled}{\Ptle d}
\newcommand{\Ptpgd}{{}^P \tau_{>d}}
\newcommand{\Dle}[1]{D^{\le#1}}
\newcommand{\Dge}[1]{D^{\ge#1}}
\newcommand{\pDle}[1]{{}^pD^{\le#1}}
\newcommand{\pDge}[1]{{}^pD^{\ge#1}}
\newcommand{\PDle}[1]{{}^PD^{\le#1}}
\newcommand{\PDge}[1]{{}^PD^{\ge#1}}

\newcommand{\m}{^{(m)}}
\newcommand{\mm}{^{(m-1)}}
\newcommand{\alb}{{\al_\beta}}
\newcommand{\Km}{K_m}
\newcommand{\Dbm}{D^b_m}
\newcommand{\Perm}{\Per_m}
\newcommand{\Fqn}{\F_{q^n}}
\newcommand{\wiZ}{{w\in\Z}}
\newcommand{\s}{_\bullet}
\newcommand{\ensdr}[2]{\left\{ #1 \,\left|\, #2 \right.\right\}}
\newcommand{\RG}{R\Gamma}

\newcommand{\red}{{\mathrm{red}}}

\DeclareMathOperator{\Spec}{Spec}
\DeclareMathOperator{\Tr}{Tr}
\DeclareMathOperator{\Fr}{Fr}
\DeclareMathOperator{\Gal}{Gal}
\DeclareMathOperator{\Ob}{Ob}
\DeclareMathOperator{\Per}{Per}
\DeclareMathOperator{\Hom}{Hom}
\DeclareMathOperator{\fr}{fr}

\let\tnm=\textnormal

\newcommand{\tm}[1]{\textnormal{(#1)}}
\newcommand{\tf}{de type fini\xspace}
\newcommand{\faisc}{faisceau\xspace}
\newcommand{\faisx}{faisceaux\xspace}
\newcommand{\cat}{catégorie\xspace}
\newcommand{\cats}{catégories\xspace}
\newcommand{\geom}{géométrique\xspace}
\newcommand{\alg}{algébrique\xspace}
\newcommand{\elt}{élément\xspace}
\newcommand{\dfn}{définition\xspace}
\newcommand{\tq}{tel que\xspace}
\newcommand{\ssi}{si et seulement si\xspace}
\newcommand{\coeff}{coefficient\xspace}
\newcommand{\appl}{application\xspace}
\newcommand{\df}{de dimension finie\xspace}
\newcommand{\ev}{espace vectoriel\xspace}
\newcommand{\resp}{
resp.\xspace}

\newcommand{\thm}{théorème\xspace}
\newcommand{\Ops}{On peut supposer\xspace}
\newcommand{\ops}{on peut supposer\xspace}
\newcommand{\irr}{irréductible\xspace}
\newcommand{\inte}{intégralité\xspace}
\newcommand{\dist}{distingué\xspace}
\newcommand{\trdists}{triangles distingués\xspace}
\newcommand{\dem}{démonstration\xspace}
\newcommand{\ie}{\emph{i. e.} }
\newcommand{\rev}{revêtement\xspace}
\newcommand{\proj}{projectif\xspace}
\newcommand{\hyp}{hypothèse\xspace}
\newcommand{\resol}{résolution\xspace}
\newcommand{\Thm}{Théorème\xspace}
\newcommand{\car}{caractéristique\xspace}
\newcommand{\cf}{\emph{cf.} }
\newcommand{\Cf}{\emph{Cf.} }
\newcommand{\enpart}{en particulier\xspace}

\newcommand{\dpr}{d'après }
\newcommand{\Dpr}{D'après }
\newcommand{\compl}{complémentaire\xspace}
\newcommand{\diacom}{diagramme commutatif\xspace}
\newcommand{\rec}{récurrence\xspace}
\newcommand{\eqdim}{équidimensionnel\xspace}
\newcommand{\gelt}{généralement\xspace}

\newcommand{\Ecom}{$E$-compatible\xspace}
\newcommand{\Ecomp}{$E$-compatibilité\xspace}
\newcommand{\Tent}{$T$-entier\xspace}
\newcommand{\Tint}{$T$-intégralité\xspace}
\newcommand{\tstr}{$t$-structure\xspace}

\newcommand{\lcit}[1][]{[\emph{ibid.}, #1]} 

\begin{document}
\title{Théorème de Gabber d'indépendance de $l$}
\author{rédigé par Weizhe \textsc{Zheng}}
\date{%
Mémoire de Master 2\ieme année\footnotemark\\
réalisé sous la direction du Professeur Luc Illusie}
\maketitle%
{\renewcommand{\thefootnote}{\fnsymbol{footnote}}%
\footnotetext[1]{Soutenu le 1\ier
juillet 2005.\\ Courriel
: \textsf{weizhe.zheng@math.u-psud.fr}}}%
{\small
\def\contentsname{Sommaire}%
\tableofcontents}

\vspace{2\baselineskip} On expose ici le \thm de Gabber
d'indépendance de $l$ pour la cohomologie d'intersection d'un
schéma propre équidimensionnel sur le spectre d'un corps fini
(\ref{th.IH}). On suit \cite{ind} à très peu près. Je remercie
vivement mon directeur, Luc Illusie, qui m'a beaucoup aidé.

\section{Suites récurrentes linéaires} On aura besoin de
certaines propriétés élémentaires des suites.

Soient $P\subset\Z$ une partie non vide stable par addition par
$\N$ et $E$ un corps. On désigne par $E^P$ l'ensemble des suites
$(a_n)_\ndP$ à valeurs dans $E$ (\ie des applications de $P$ vers
$E$). On pose
\begin{displaymath}
\begin{array}{rrcl}
T:& E^P &\to& E^P\\
&(a_n)_\ndP \mt (a_{n+1})_\ndP.
\end{array}
\end{displaymath}

\begin{prop}\label{pr.Rl}
Soit $a=(a_n)_{n\in P} \in E^P$. Les conditions suivantes sont
équivalentes :

\tm{a} Il existe un polynôme\/ $Q(X) \in E[X]$ à \coeff constant
non-nul tel que\/ $Q(T)a=0$.

\tm{b} il existe $r\ge 0$, $c_0,\dots,c_r\in E$, $c_0c_r \neq 0$
tels que
\[\sum_{i=0}^r c_i a_{n+i}=0 \text{ pour tout } n\in P.\]

\tm{c} le sous-$E$-\ev\/ $E[T]a\subset E^P$ est \df et\/ $T$ en
est un automorphisme.
\end{prop}

\begin{proof}
(a) $\Leftrightarrow$ (b) Si on pose $Q(X)=\sum_{i=0}^r c_i X^i$,
alors $(Q(T)a)_n=\sum_{i=0}^r c_i a_{n+i}$.

(a) et (b) $\Rightarrow$ (c) $E[T]a$ est un sous-$E$-\ev \df, \dpr
(a). L'action de $T$ en est injective, \dpr (b).

(c) $\Rightarrow$ (a) On peut prendre comme $Q(X)$ le polynôme
mininal de $T$ sur $E[T]a$.
\end{proof}

\begin{defn}
Une suite $a\in E^P$ est dite récurrente linéaire si elle
satisfait aux conditions équivalentes de \ref{pr.Rl}. L'ensemble
des suites récurrentes linéaires se note $\Rl(P,E)$.
\end{defn}

Pour tout $\alpha\in E^*$, $(\alpha^n)_\ndP\in \Rl(P,E)$.

\begin{coro}\label{cor.Rl}
\tm{i} $\Rl(P,E)$ est un sous-$E$-\ev de\/ $E^P$.

\tm{ii} L'\appl de restriction\/ $\Rl(\Z,E)\to \Rl(P,E)$ est
bijective.

\tm{iii} Soit\/ $E'$ une extension de\/ $E$. Si $a\in \Rl(P,E')$
prend toutes ses valeurs dans\/ $E$, alors $a\in \Rl(P,E)$.
\end{coro}

\begin{proof}
(i) (\resp (ii)) est clair d'après la condition (a) (\resp (b)).

Pour (iii), on utilise la condition (c). On a $E[T]a \otimes_E E'
\simeq E'[T]a$, donc $E[T]a$ est \df sur $E$. On a une injection
$E[T]a \har E'[T]a$, ce qui donne l'injectivité de l'action de $T$
sur $E[T]a$.
\end{proof}

\section{Notations et rappels}
Pour tout nombre premier, on choisit une clôture algébrique
$\Qlb$. Soit $F$ un corps de \car $\chi_F \neq l$. Pour un schéma
$X$ séparé \tf sur $\Spec F$, on définit les $\Qlb$-faisceaux
constructibles sur $X$ et la catégorie dérivée
$\Dbc(X,\Qlb)$ (\cf 
\cite[1.1]{WeilII} et \cite{Ekedahl}). Les \cats dérivées
$\Dbc(-,\Qlb)$ sont stables par les six opérations $Rf_!, Rf_*,
f^*, Rf^!, \Lt, \RcHom$, et \enpart par le foncteur dualisant $D$.
On normalise le dernier par
\[D_X K = \RcHom(K,Ra^!\Qlb)\]
où $a: X\to\Spec F$.

On désigne par $K(X,\Qlb)$ le groupe de Grothendieck des
$\Qlb$-\faisx. Pour un $\Qlb$-faisceau $\cF$ sur $X$, on note
$[\cF]$ sa classe dans $K(X,\Qlb)$. L'application $\cF\mapsto
[\cF]$ se prolonge en une \appl surjective
\begin{ea}
\Ob(\Dbc(X,\Qlb)) &\to& K(X,\Qlb)\\
K\mt [K]=\sum_{q\in\Z} (-1)^q [\fH^q(K)].
\end{ea}
On désigne parfois l'image de $K$ encore par $K$. Les six
opérations sur $\Dbc(-,\Qlb)$ induisent les 
homomorphismes entre les groupes $K(-,\Qlb)$.

\subsection{Perversité}

On prend la perversité autoduale $p$ et la \tstr $(\pDle0,\pDge0)$
sur $\Dbc(X,\Qlb)$ comme dans \cite[4.0]{BBD}. On rappelle que
pour $K\in \Dbc(X,\Qlb)$,
\[
K\in \pDle0 \iff \fH^q(i_x^*K)=0, \forall x\in X, \forall
q>-\delta(x),
\]
\[
K\in \pDge0 \iff \fH^q(i_x^!K)=0, \forall x\in X, \forall
q<-\delta(x),
\]
où $i_x$ est l'inclusion de $\{x\}$ dans $X$, $i_x^!=j^*Ri^!$ pour
$\{x\} \sto[j] \adhx \stackrel{i}{\hookrightarrow} X$,
$\delta(x)=\dim\adhx$ \lcit[2.2.12 et 2.2.19].

Le foncteur dualisant $D_X$ échange $^p D^{\le 0}$ et $\pDge0$. Si
on note $d=\dim X$, alors $\Dle{-d} \subset \pDle0 \subset \Dle0$,
donc $\Dge0 \subset\pDge0 \subset \Dge{-d}$ par l'orthogonalité \lcit[1.3.4].

On note $\Per(X,\Qlb)$ la \cat des \faisx pervers sur $X$. Tout
objet de cette \cat est de longueur finie \lcit[4.3.1(i)], donc
\beq\label{eq.K=K(P)}
K(X,\Qlb)=K(\Per(X,\Qlb))
\eeq
est le groupe abélien libre
engendré par les classes d'isomorphie des \faisx pervers simples.

Soit $j: U \har X$ une immersion ouverte. Alors l'extension
intermédiaire d'un faisceau pervers $K\in \Per(U,\Qlb)$ par $j$ se
calcule par la formule $j_{!*}K = {}^p \tau^Y_{\le-1} Rj_*K$
\lcit[1.4.23], où $Y\stackrel{i}{\hookrightarrow} X$ est le fermée
\compl. Si $f: K_1 \to K_2$ est un morphisme dans $\Per(U,\Qlb)$,
alors $j_{!*}f = \ptle[Y]{-1} Rj_* f$.\footnote{Ceci découle de la
factorisation du \diacom dans $\Per(X,\Qlb)$
\[\xymatrix{{}^pj_!K_1 \ar[d]_{{}^pj_! f} \ar[r] & {}^pRj_* K_1\ar[d]^{{}^pRj_* f}\\
{}^pj_!K_2 \ar[r] &{}^pRj_* K_2}\] en le diagramme
\[\xymatrix{{}^pj_!K_1 \ar[d]_{{}^pj_! f} \ar@{>>}[r] &
\ptYle{-1} Rj_* K_1\ar[d]_{\ptYle{-1} Rj_* f} \ar@{>->}[r]
& {}^pRj_* K_1\ar[d]^{{}^pRj_* f}\\
{}^pj_!K_2 \ar@{>>}[r] & \ptYle{-1} Rj_* K_1 \ar@{>->}[r]
&{}^pRj_* K_2}\] dont le carré à gauche est commutatif car celui à
droite l'est.}

Notons que
\[\ensdr{K\in\Dbc(X,\Qlb)}{i^*K\in \Dle{-d-1}_Y} \subset
\ensdr{K\in\Dbc(X,\Qlb) }{ i^*K\in \pDle{-1}_Y } ,\] où $d=\dim
Y$. On a donc un morphism de foncteurs
\[\tYle{-d-1} \to \ptYle{-1}
: \Dbc(X,\Qlb) \to \Dbc(X,\Qlb).\]

\begin{prop}\label{pr.j!*}
\tm{i} Si le morphisme $\tYle{-d-1} Rj_* K_\beta \to \ptYle{-1}
Rj_* K_\beta$ est un isomorphisme pour chaque\/ $\beta=1,2$, alors
on peut faire des identifications $j_{!*} K_\beta = \tYle{-d-1}
Rj_* K_\beta$, $\beta=1,2$ et $j_{!*} f = \tYle{-d-1} Rj_* f$.

\tm{ii} Si $Y$ est lisse purement de dimension $d$ et les $\fH^e
i^*Rj_*K_\beta$, $e \ge -d$, sont lisses pour\/ $\beta=1,2$, alors
l'hypothèse de \tm{i} est vraie.
\end{prop}

\begin{proof}
(i) On a un diagramme commutatif
\[\xymatrix{\tYle{-d-1} Rj_* K_1 \ar[d]_{\tYle{-d-1} Rj_* f}\ar[r] & \ptYle{-1}
Rj_* K_1\ar[d]^{\ptYle{-1}
Rj_* f}\\
\tYle{-d-1} Rj_* K_2 \ar[r] & \ptYle{-1} Rj_* K_2}\]

(ii) \lcit[2.2.4 et 2.2.19]
\end{proof}

Dans le cas général, on peut calculer $j_{!*}$ en itérant
\ref{pr.j!*}.

\begin{coro}\label{cor.j!*}
Supposons\/ $F$ parfait. Soient $A$ un ensemble fini,
$(l_\alpha)_{\alpha\in A}$ une famille de nombres premiers $\neq
\chi_F$, $(K_\al)_\aldA \in \prod_{\aldA} \Per(U,\Qlalb)$. Alors
il existe $n\ge 0$ et des immersions ouvertes
\[U=U_0 \har[j_1] U_1 \har[j_2] \cdots \har[j_n] U_n=X\]
vérifiant les conditions suivantes

\tm{i} Pour $1\le m \le n$, $Y_m = U_m \bsl U_{m-1}$ (équipé de la
structure de sous-schéma réduit induite) est lisse sur\/ $F$ et
purement de dimension $d_m$, $d_1 > d_2 > \cdots > d_n \ge 0$,
$K_\al\m = \tle[Y_m]{-d_m-1} Rj_{m*} K_\al\mm$, et $\fH^e(i_m^!
K_\al\m)$ est lisse sur $Y_m$, $\forall e\in \Z, \aldA$, où $i_m:
Y_m \har U_m$, $K_\al\m = j_{m!*} K_\al\mm$, $K_\al^{(0)} =
K_\al$. Donc
\[j_{!*}K_\al = K_\al^{(n)} = \tle[Y_n]{-d_n -1}Rj_{n*} \cdots \tYle[1]{-d_1 -1} Rj_{1*} K_\al,
\quad \forall \aldA.\]

\tm{ii} Soient $\al_1, \al_2 \in A$ tels que $l_{\al_1} =
l_{\al_2}$ et $f:K_{\al_1} \to K_{\al_2}$ un morphisme. Alors
\[j_{!*}f = \tle[Y_n]{-d_n -1}Rj_{n*} \cdots \tYle[1]{-d_1 -1} Rj_{1*}
f.\]

\tm{iii} Si, en outre, $U$ est lisse sur\/ $F$ purement de
dimension $d>\dim(X-U)$ et $\fH^e(K_{\al_\beta})$ est lisse sur\/
$U$, $\forall e\in \Z$, $\beta=1,2$, alors\/ $\forall \beta=1,2$,
$K_\alb\m$ est concentré en degrés $[-d,-d_m-1]$ pour $1\le m \le
n$, et
\begin{ea}
j_{!*}K_\alb &=& \tle{-d_n -1}Rj_{n*} \cdots
\tle{-d_1 -1} Rj_{1*} K_\alb, \\
j_{!*}f &=& \tle{-d_n -1}Rj_{n*} \cdots \tle{-d_1 -1} Rj_{1*} f.
\end{ea}
\end{coro}

La condition (iii) ne sera pas utilisée dans la suite.

\begin{proof}
(\Cf \cite[1.5 ?]{Tian}) On montre (i) et (ii) en même temps.
(iii) en suit.

\Ops $U\subsetneqq X$. Par \rec sur $\dim(X-U)$, il suffit de
montrer qu'il existe des immersions ouvertes
\[U\har[j_1] U_1 \har[j'] X\]
telles que $Y_1 = U_1\bsl U$ (réduit) soit lisse sur $F$ et
purement de dimension $d_1 > \dim (X-U_1)$, $j_{1!*} K_\al =
\tYle[1]{-d_1-1} Rj_{1*} K_\al$, $j_{1!*} f = \tYle[1]{-d_1-1}
Rj_{1*} f$, et que $\fH^e(i_1^! j_{1!*} K_\al)$ soit lisse sur
$Y_1$, $\forall e\in \Z$, $\aldA$, où $i_1: Y_1 \har U_1$.

On pose $Y=X\backslash U$ et lui donne la structure de schéma
réduit induite. On prend un ouvert $W$ de $Y$ \eqdim tel qu'il
soit lisse sur $F$, $\dim(Y-W) <\dim Y$ et que $\fH^e(j_W^* i^*
Rj_*K_\al)$, $\fH^e(j_W^* Ri^! j_{!*}K_\al)$ soient lisses sur $W$
pour tout $e\in \Z$ et tout $\aldA$, où $i: Y\har X$, $j_W: W\har
Y$. On pose $U_1= U\cup W$. Alors $Y_1 = W$ est lisse sur $F$ et
purement de dimension $d_1=\dim Y >\dim (X-U_1)$, car $X-U_1 =
Y-W$. On a le \diacom d'immersions
\[\xymatrix{U \ar
@(ur,ul)[rr]^j \ar
[r]^{j_1} & U_1 \ar
[r]^{j'} &X\\
& Y_1 \ar
[u]^{i_1} \ar
[r]^{j_W} &
Y\ar
[u]^{i}}\] Notons que
\[j_W^*i^*Rj_* = i_1^* j'^* Rj'_{*} Rj_{1*}= i_1^*
Rj_{1*},\]
\[j_W^*Ri^!j_{!*} = Ri_1^! j'^* j'_{!*} j_{1!*} = Ri_1^!
j_{1!*}.\]

Donc $\fH^e(i_1^*Rj_{1*}K_\al)$, $\fH^e(i_1^! j_{1!*} K_\al)$ sont
lisses sur $Y_1$, $\forall e\in \Z$, $\aldA$. D'après
\ref{pr.j!*}, $j_{1!*} K_\al = \tYle[1]{-d_1-1} Rj_{1*} K_\al$,
$\aldA$ et $j_{1!*} f = \tYle[1]{-d_1-1} Rj_{1*} f$.
\end{proof}

\subsection{} 
On fixe un corps fini $k=\F_q$, $q=p^\nu$, et une clôture \alg
$\kb$. Désormais, sauf mention expresse du contraire, on travaille
sur $k$ et on ne considère que des schémas séparés \tf sur $\Spec
k$. Pour un tel schéma $X$, on désigne par $| X |$ l'ensemble des
points fermés de $X$. Soit $l$ un nombre premier $l\nmid q$. Pour
un $\Qlb$-\faisc $\cF$ sur $X$ et $x\in |X|$ (avec un point \geom
\alg $\xb$ de $X$ localisé en $x$), le Frobenius \geom
$\Fr_x\in\Gal(k(\xb)/k(x))$ agit sur la fibre \geom $\cF_{\xb}$
par transport de structure. On pose
\[L_x(\cF,t) = \det(1-t^{\deg x}\Fr_x, \cF_{\xb})^{-1} \in \Qlb(t).\]
On peut le voir comme un \elt de $1+t\Qlb [[t]]$, et on a alors
\beq\label{eq.Lx}
L_x(\cF,t) = \exp(\sum_{n\ge 1} \Tr(\Fr_x^n,\cF_\xb) t^{n\deg
x}/n).
\eeq
La \dfn de $L_x(\cF,t)$ et la formule (\ref{eq.Lx}) s'étendent à
$K(X,\Qlb)$ par additivité, et l'homomorphisme 
\begin{ea}
K(X,\Qlb) &\to& \prod_{x\in |X|} (1+t\Qlb [[t]])\\
K \mt (L_x(K,t))_{x\in |X|}
\end{ea}
est injectif \cite[1.1.2]{Laumon}.

\subsection{Pureté}

On renvoie à \cite[6.2]{WeilII} pour la notion de complexe pur
(\resp mixte). Soit $w$ un entier. On rappelle que les complexes
mixtes de poids $\le w$ sont stables par $f^*$ et $Rf_!$, et que
les complexes mixtes de poids $\ge w$ sont stables par $Rf^!$ et
$Rf_*$.

On note $\Dbm$ la \cat des complexes mixtes. On note
$\Perm
$ (\resp $\Per_w
$) la \cat des \faisx
pervers mixtes (\resp purs de poids $w$), et $\Km
$ (\resp
$K_w
$) son groupe de Grothendieck. Alors $\Km$ est aussi le groupe de
Grothendieck de $\Dbm$. D'après \cite[5.1.7(ii) (\resp
5.3.1)]{BBD}, $\Perm
$ (\resp $\Per_w
$) est une sous-\cat épaisse de $\Per
$, donc $\Km
$ (\resp
$K_w
$) est le sous-groupe abélien libre de $K
$ (\ref{eq.K=K(P)}) engendré par les classes d'isomorphie des
\faisx pervers simples
mixtes (\resp purs de poids $w$). Par conséquent $\Km
=
\oplus_{w\in\Z}
K_w
$. Plus généralement, pour tout
intervalle (éventuellement non borné) $I\subset \Z$, on peut définir $\Per_I
$ et
$K_I
$. On a $K_I
= \oplus_{w\in I}K_w
$.

\begin{theo}[Gabber]\label{th.pHi}
Soit\/ $K\in \Dbm(X,\Qlb)$. Pour que\/ $K$ soit de poids $\le w$
(\resp $\ge w$), il faut et il suffit que chaque\/ $\pH^iK$ soit
de poids $\le w+i$ (\resp $\ge w+i$). \tnm{\lcit[5.4.1]}
\end{theo}

\begin{coro}[Gabber]\label{th.Gabber}
Soit $f$ un morphisme quasi-fini de schémas séparés \tf sur\/
$\Spec k$. Alors $f_{!*}$ préserve\/ $\Per_I(-,\Qlb)$.
\tnm{\lcit[5.4.3]}
\end{coro}

\begin{coro}\label{cor.pds}
On garde les notations et les hypothèses de \tnm{\ref{cor.j!*}}
(pour\/ $F=k$). Soit $\aldA$.

\tm{i} Si\/ $K_\al \in \Per_{\le w}(U,\Qlalb)$, alors pour\/ $1\le
m \le n$, $\fH^e(Rj_{m*}K_\al\mm)$ est mixte de poids ponctuels
$\le w-d_m-1$, $\forall e\le -d_m-1$.

\tm{ii} Si\/ $K_\al \in \Per_{\ge w}(U,\Qlalb)$, alors pour\/
$1\le m \le n$, $\fH^e(i_m^* Rj_{m*}K_\al\mm)$ est mixte de poids
ponctuels\/ $\ge w-d_m+1$, $\forall e\ge -d_m$.
\end{coro}

\begin{proof}
$K_\al\m = \tYle[m]{-d_m-1} Rj_{m*}K_\al\mm$.

(i) D'après \ref{th.Gabber}, $K_\al\m \in \Per_{\le w}(U_m,
\Qlalb)$. Pour $e \le -d_m-1$,
\[\fH^e(Rj_{m*}K_\al\mm) =
\fH^e(K_\al\m)\] est mixte de poids $\le w+e \le w-d_m-1$.

(ii) D'après \ref{th.Gabber}, $K_\al\m \in \Per_{\ge w}(U_m,
\Qlalb)$. On a un triangle \dist
\[Ri_m^! K_\al\m \to i_m^* K_\al\m \to i_m^*Rj_{m*} K_\al\mm \to,\]
déduit du triangle \dist $Ri_m^!M\to i_m^*M \to
i_m^*Rj_{m*}j_m^*M\to$ pour $M\in \Dbc(U_m,\Qlalb)$. Pour $e\ge
-d_m$, $\fH^e(i_m^*Rj_{m*}K_\al\mm) =\fH^{e+1}(Ri_m^! K_\al\m)$ est
donc mixte de poids ponctuels $\ge w+e+1\ge w-d_m+1$, car $Ri_m^!
K_\al\m$ est mixte de poids $\ge w$ et à \faisx de cohomologie
lisses.
\end{proof}

\begin{lemm}\label{lm.t}
Soient\/ $D$ une \cat triangulée munie d'une \tstr $(\PDle0,
\PDge0)$, $K\sto[f] L \sto[g] M \sto[h]$ un triangle \dist dans\/
$D$, $d\in \Z$ \tq le cobord\/ $\PH^d(h) : \PH^d(M) \to
\PH^{d+1}(K)$ soit nul. Alors on a un diagramme des 9
\beq\label{9dia.1}
\xymatrix{\Ptled K \ar[d]\ar[r]^{\Ptled f} &\Ptled L
\ar[d]\ar[r]^{\Ptled g} &\Ptled M\ar[d]\ar[r] &\\
K\ar[d]\ar[r]^f & L\ar[d]\ar[r]^g & M\ar[d]\ar[r]^h &\\
\Ptpgd K\ar[d]\ar[r]^{\Ptpgd f} & \Ptpgd L\ar[d]\ar[r]^{\Ptpgd g} & \Ptpgd M\ar[d]\ar[r] &\\
&&&}
\eeq
dont les colonnes sont des triangles distingués canoniques.
\end{lemm}

\begin{proof}
On complète le \diacom
\[\xymatrix{\Ptled K \ar[d]\ar[r]^{\Ptled f} &\Ptled L\ar[d]\\
K\ar[r]^f &L}\] en un diagramme des 9
\beq\label{9dia.2}
\xymatrix{\Ptled K \ar[d]\ar[r]^{\Ptled f} &\Ptled L
\ar[d]\ar[r] & M_1\ar[r]\ar[d] &\\
K\ar[d]\ar[r]^f & L\ar[d]\ar[r]^g & M\ar[d]\ar[r]^h &\\
\Ptpgd K\ar[d]\ar[r] & \Ptpgd L\ar[d]\ar[r] & M_2\ar[d]\ar[r] &\\
&&&}
\eeq
dont les deux premières colonnes sont des trianlges distingués
canoniques. La première (\resp troisième) ligne implique que
$M_1\in \PDle d$ (\resp $M_2\in \PDge d$). Le diagramme des suites
exactes longues donne alors un \diacom
\[\xymatrix{&\PH^d M\ar[d]\ar[r]^{
0}& \PH^{d+1}K\ar[d]\\
0\ar[r]& \PH^d M_2 \ar[d]\ar[r] &\PH^{d+1}K \\
&0}\] qui implique $\PH^d M_2 =0$, donc $M_2\in {}^PD^{>d}$. En
utilisant \cite[1.1.9]{BBD}, on voit alors que (\ref{9dia.2})
s'identifie à un diagramme de la forme (\ref{9dia.1}).
\end{proof}

\begin{coro} Soit $f:Z\to X$ un morphisme quasi-fini. Alors le
foncteur
\[f_{!*} : \Per_{[w,w+1]}(Z,\Qlb) \to \Per_{[w,w+1]}(X,\Qlb)\]
est exact.
\end{coro}

\begin{proof}
Le cas $f$ fini étant trivial, on peut supposer que $f$ soit une
immersion ouverte $j:U\har X$. On se donne une suite exacte courte
$0\to K_1 \sto[g_1] K_2 \sto[g_2] K_3 \to 0$ dans
$\Per_{[w,w+1]}(U,\Qlb)$, et on cherche à montrer que la suite
déduite $0\to j_{!*}K_1 \to j_{!*}K_2 \to j_{!*}K_3 \to 0$ est
exacte.

On prend des $U_m$, $0\le m \le n$, comme dans \ref{cor.j!*} et on
montre l'exactitude de
\beq\label{eq.se}
0\to K_1\m \sto[g_1\m] K_2\m \sto[g_2\m] K_3\m \to 0
\eeq
par récurrence sur $m$. Le cas $m=0$ est tautologique. On suppose
l'exactitude de (\ref{eq.se}) établie pour $m-1$, $1\le m\le n$.
Alors on a un triangle \dist
\[Rj_{m*} K_1\mm \sto[Rj_{m*} g_1\mm] Rj_{m*} K_2\mm \sto[Rj_{m*} g_2\mm]
Rj_{m*} K_3\mm \sto[h].\]
\Dpr \ref{cor.j!*},
\begin{ea}
K_\al\m &=& \tYle[m]{-d_m-1} Rj_{m*} K_\al\mm, \al=1,2,3,\\
g_\al\m &=& \tYle[m]{-d_m-1} Rj_{m*} g_\al\mm, \al=1,2.
\end{ea}
Par \dfn, $\tYle[m]{-d_m-1}= \Ptle{-d_m-1}$, où $P$ est la \tstr
sur $U_m$ obtenue par recollement des $t$-structures
$(\Dbc(U_m,\Qlb),0)$ et \[(\Dle0(Y_m,\Qlb),\Dge0(Y_m,\Qlb)).\]
Notons \lcit[1.4.13]
\begin{ea}
\PH^{-d_m-1}(Rj_{m*}K_3\mm) &=&
i_{m*}\fH^{-d_m-1}(i_m^*Rj_{m*}K_3\mm), \\
\PH^{-d_m}(Rj_{m*}K_1\mm) &=&
i_{m*}\fH^{-d_m}(i_m^*Rj_{m*}K_1\mm).
\end{ea}
\Dpr \ref{cor.pds},
$\fH^{-d_m-1}(i_m^* Rj_{m*} K_3\mm)$ est mixte de poids ponctuels
$\le w-d_m$, $\fH^{-d_m}(i_m^* Rj_{m*} K_1\mm)$ est mixte de poids
ponctuels $\ge w-d_m+1$. Donc $\PH^{-d_m-1}(h)=0$. \Dpr
\ref{lm.t}, on a alors un triangle \dist
\[K_1\m \sto[g_1\m] K_2\m \sto[g_2\m] K_3\m \to,\]
qui donne l'exactitude de (\ref{eq.se}).
\end{proof}

\begin{remq}\label{rm.[f!*]}
L'opération $f_{!*}$ sur $\Per_w(-,\Qlb)$ induit donc un
homomorphisme des groupes de Grothendieck $K_w(-,\Qlb)$, et on
définit un homomorphisme $f_{!*}$ sur $\Km(X,\Qlb)$ en prenant la
somme directe. On a alors un \diacom
\[\xymatrix{\Ob(\Per_{[w,w+1]}(Z,\Qlb))\ar[d]\ar[r]^{f_{!*}} &
\Ob(\Per_{[w,w+1]}(X,\Qlb))\ar[d]\\
\Km(Z,\Qlb)\ar[r]^{f_{!*}} & \Km(X,\Qlb)}
\]

Quand $f$ est fini, cette \dfn coïncide avec le $f_*$ virtuel
restreint à $\Km$.
\end{remq}

\begin{prop}\label{pr.dec1}
Soient\/ 
$K\in \Per_w(X,\Qlb)$, $i: Y\hookrightarrow X$ une immersion
fermée et $j: U \har X$ l'ouvert \compl. Alors\/ $K$ admet une
unique décomposition\/ $K=j_{!*} K'\oplus i_* K''$, où\/ $K'\in
\Per_w(U,\Qlb)$, $K''\in\Per_w(Y,\Qlb)$. \tnm{\lcit[5.3.11]}
\end{prop}

\begin{prop}\label{pr.Extnul}
Soient\/ $K,L\in \Dbm(X,\Qlb)$ purs de poids $w$. Alors le
morphisme\/ $\Hom(K,L[1]) \to \Hom(K_\kb,L_\kb[1])$ est nul.
\tnm{\lcit[5.1.15(iii)]}
\end{prop}

\begin{coro}\label{cor.dec}
Soit\/ $K\in \Dbm(X,\Qlb)$ pur de poids $w$. Alors
\[\det(1-t\Fr, H^i(X_\kb,K_\kb)) =
\prod_{e\in \Z} \det\!\left(1-t\Fr, H^i(X_\kb,
(\pH^eK_\kb)[-e])\right).\]
\end{coro}

\begin{proof}
Pour tout $d\in \Z$, on a un morphisme de \trdists
\[\xymatrix{\fr^* \ptle{d} K_\kb\ar[d]_{\Fr}\ar[r]& \fr^* K_\kb \ar[d]_{\Fr}\ar[r] &
\fr^* \ptpg d K_\kb\ar[d]_{\Fr} \ar[r]^(.7)0 &\\
\ptle d K_\kb \ar[r]& K_\kb\ar[r] &\ptpg d K_\kb\ar[r]^(.7)0 &}\]
où la nullité des flèches de degrés $1$ découle de \ref{th.pHi} et
de \ref{pr.Extnul}. Alors on a un morphisme de \trdists
\[\xymatrix{\RG(\Xkb, \ptle{d} K_\kb)\ar[d]_{\Fr}\ar[r]& \RG(\Xkb, K_\kb) \ar[d]_{\Fr}\ar[r] &
\RG(\Xkb, \ptpg d K_\kb)\ar[d]_{\Fr} \ar[r]^(.7)0 &\\
\RG(\Xkb,\ptle d K_\kb) \ar[r]& \RG(\Xkb,K_\kb)\ar[r]
&\RG(\Xkb,\ptpg d K_\kb)\ar[r]^(.7)0 &}\] Donc
\begin{multline*}
\det(1-t\Fr, H^i(X_\kb,K_\kb)) \\
= \det(1-t\Fr, H^i(X_\kb,  \ptle{d} K_\kb)) \det(1-t\Fr,
H^i(X_\kb, \ptpg{d} K_\kb)).
\end{multline*}
\end{proof}

\section{$(E,I)$-compatibilité}
Soit $E$ un corps, $I$ une partie de
\[\ensdr{\lio}{ l\nmid q
\text{ un nombre premier, } \iota: E\hookrightarrow\Qlb \text{ un
plongement de corps}}.\]

\begin{defn}
On dit qu'un système $(t_\lio)_\lidI \in \prod_\lidI \Qlb$ est
$(E,I)$-compatible (ou $E$-compatible s'il n'y a pas de confusion
à craindre) s'il existe $c\in E$ \tq $t_\lio = \iota(c)$ pour tout
$\lio\in I$.
\end{defn}

\begin{remq}\label{rm.Rl}
Si $P$ est comme dans \S 1, $(a_\lio)_\lidI \in \prod_\lidI
\Rl(\Z,\Qlb)$ \tq $(a_{\lio,n})_\lidI$ $E$-compatible $\forall
n\in P$, alors $(a_{\lio,n})_\lidI$ est $E$-compatible $\forall
n\in \Z$, \dpr \ref{cor.Rl}(ii) et (iii).
\end{remq}

\begin{defn}
Soit $X$ un schéma séparé \tf sur $k=\F_q$. On dit qu'un système
$(K_\lio)_\lidI \in \prod_\lidI K(X,\Qlb)$ est $(E,I)$-compatible
(ou $E$-compatible) si pour tout $x\in|X|$ et tout $n\ge 1$,
$(\Tr(\Fr_x^n, (K_\lio)_\xb))_\lidI \in \prod_\lidI \Qlb$ est
$(E,I)$-compatible.
\end{defn}

D'après (\ref{eq.Lx}), $(K_\lio)$ est $E$-compatible \ssi pour
tout $x\in |X|$, il existe $s_x\in E[[t]]$ \tq $L_x(K_\lio,t) =
\iota(s_x)$ pour tout $\lidI$. Ici on a étendu $\iota$ en un
plongement d'anneaux $E[[t]]\to \Qlb [[t]]$.

Les systèmes $E$-compatibles forment un sous-anneau de
$\prod_\lidI K(X,\Qlb)$.

\begin{exem}
$(\Qlb)_\lidI$ est \Ecom. Plus \gelt, pour $b\in E^*$ \tq
$\iota(b)$ soit une unité $l$-adique $\forall \lidI$, le système
$(\Qlb^{(\iota(b))})_\lidI$ \cite[1.2.7]{WeilII} est \Ecom, car
les traces locales sont $\iota(b)$.
\end{exem}

\subsection{Stabilités}

La $E$-compatibilité est stable par les six opérations et donc par
le foncteur dualisant.

\begin{theo}\label{th.6op}
Soient $f:X\to Y$ un morphisme de schémas séparés \tf sur $k$ et
$(K_\lio)_\lidI\in \prod_\lidI K(X,\Qlb)$ un système
$(E,I)$-compatible sur $X$. Alors $(Rf_*K_\lio)_\lidI,
(Rf_!K_\lio)_\lidI \in \prod_\lidI K(Y,\Qlb)$ sont des systèmes
$(E,I)$-compatibles sur $Y$. On a des résultats similaires pour
$f^*, Rf^!$ et pour $\Lt, \RcHom, D$.
\end{theo}

\begin{proof}[Esquisse de la démonstration]
Les résultats pour $\Lt$ et $f^*$ sont triviaux. Pour $Rf_!$ on
utilise la formule des traces
\[\Tr(\Fr_y,(Rf_! K_\lio)_\yb) = \sum_{x\in X_y(\F_{q^n})} \Tr(\Fr_x, (K_\lio)_\xb),
\forall y\in Y(\F_{q^n}), n\ge 1.\] Il reste à montrer le résultat
pour $D$.

On prend un système $(K_\lio)_\lidI$ \Ecom. Il suffit de voir la
\Ecomp de $(\Tr(\Fr_x,(DK_\lio)_\xb))_\lidI$ pour chaque $x\in
X(\F_{q^n})$, $n\ge 1$. Le problème est local. Par dévissage, \ops
que $X=A$ soit une variété abélienne et $x=0_A \in A(k)$ soit
l'origine. On définit $f_\lion, g_\lion : A(k)\to \Qlb$ pour $n\ge
1$ par
\begin{ea}
f_\lion(a) &=& \sum_{b\in A(\Fqn), T_n(b)=a} \Tr(\Fr_b,
(K_\lio)_\bb),\\
g_\lion(a) &=& \sum_{b\in A(\Fqn), T_n(b)=-a} \Tr(\Fr_b,
(DK_\lio)_\bb),
\end{ea}
où $T_n : A(\Fqn) \to A(\F_q)$ est la trace. Il suffit de voir la
\Ecomp de $(g_{\lio,1}(0))$.

On sait la \Ecomp de $(f_\lion(0))_\lidI$ par hypothèse et on veut
démontrer la \Ecomp de $(g_\lion(0))_\lidI$. \Dpr \ref{rm.Rl}, il
suffit donc de trouver pour chaque $\lidI$ une $s_\lio\in
\Rl(\Z,\Qlb)$ telle que $s_\lion = f_\lion(0)$, $s_{\lio,-n} =
g_\lion(0)$ pour tout $n\ge 1$. $\lio$ étant fixé, on ne l'indique
plus dans les indices.

Pour une fonction $f:A(k)\to \Qlb$, on définit la transformée de
Fourier par
\[\cF(f)(\rho) = \sum_{a\in A(k)} f(a)\rho (a),\]
où $\rho: A(k) \to \Qlb^*$ est un caractère
. On va montrer que pour tout $\rho$, il existe $(S_n(\rho))_{n\in
\Z} \in \Rl(\Z,\Qlb)$ telle que $S_n(\rho)=\cF(f_n)(\rho)$,
$S_{-n}(\rho)= \cF(g_n)(\rho)$ pour $n\ge 1$. On prend $s_n =
\cF^{-1}(S_n)(0)$ pour tout $n\in\Z$. Alors $s\in \Rl(\Z,\Qlb)$ \dpr
\ref{cor.Rl}(i).

En effet, par la formule des traces,
\[\cF(f_n)(\rho) = \Tr(\Fr^n, R\Gamma(A_\kb, K\Lt \fL_\rho)),\]
\[\cF(g_n)(\rho) = \Tr(\Fr^n, R\Gamma(A_\kb, DK\Lt \fL_{\rho^{-1}})),\]
où $\fL_\rho$ est le $\Qlb$-faisceau lisse de rang 1 correspondant à
$\rho$ \cite[Sommes trig.]{SGA4d}. Le résultat découle donc de la
dualité entre $R\Gamma(A_\kb, K\Lt \fL_\rho)$ et \[R\Gamma(A_\kb,
DK\Lt \fL_{\rho^{-1}}).\]
\end{proof}

Pour les détails, on renvoie à \cite[\S 3]{ind}.

\begin{exem}
Pour tout schéma $X$ séparé \tf sur $k$,
\[\Tr(\Fr,R\Gamma(X_\kb,\Q_l))\]
est dans $\Q$ (et même dans $\Z$, \cf \ref{th.int}) et indépendant
de $l$.
\end{exem}

\begin{exem}
On prend $E=\Q(\zeta_p)$ et on choisit un $I$. On fixe un
caractère additif non-trivial $\psi_0: \F_p \to E^*$ et on pose
$\psi_\lio: \iota \circ \psi_0 \circ \Tr_{\F_q/\F_p} : k \to
\Qlb^*$, $\lidI$. On considère $A=\mathbb{A}^1_k$ et
$\fL_{\psi_\lio}$ sur $A$. Alors $(\fL_{\psi_\lio})_\lidI$ est un
système \Ecom.

Soient $V=\mathbb{A}^d_k$, $V'$ son dual, $f: V\times V'\to A$
l'accouplement canonique, et $(K_\lio)_{\lidI} \in \prod_\lidI
\Dbc(V,\Qlb)$ un système \Ecomp. \Dpr \ref{th.6op}, les
transformées de Fourier-Deligne
\[\cF_{\psi_\lio}K_\lio = Rp_{2!} (p_1^* K_\lio \Lt f^*\fL_{\psi_\lio})[d]\]
forment un système \Ecom sur $V'$, où $p_1:V\times V'\to V$, $p_2:
V\times V'\to V'$ sont des projections.
\end{exem}

La $E$-compatibilité est également stable par l'extension
intermédiaire virtuelle (\ref{rm.[f!*]}) par un morphisme quasi-fini
pour les $\Km$.

\begin{theo}\label{th.f!*}
Soient $f:Z\to X$ un morphisme quasi-fini de schémas séparés \tf sur
$k$, $(K_\lio)_\lidI\in \prod_\lidI \Km(Z,\Qlb)$ un système\/
$E$-compatible sur\/ $Z$. Alors $(f_{!*}K_\lio)_\lidI\in \prod_\lidI
\Km(X,\Qlb)$ est un système\/ $E$-compatible sur\/ $X$.
\end{theo}

\begin{proof} D'après \ref{th.6op}, \ops que $f$ soit une immersion ouverte
$j:U\har X$. \Ops $\sharp I=1$ ou $2$.

(a) Cas où il existe $w\in \Z$ vérifiant $K_\lio\in K_w(U,\Qlb)$
pour tout $\lidI$. On écrit $K_\lio= [L_{\lio,1}]-[L_{\lio,2}]$, où
$L_{\lio,\al}\in \Per_w(U,\Qlb)$, $\al=1,2$, $\lidI$. On applique
\ref{cor.j!*} (pour $A=I\times\{1,2\}$) et on pose
$K_\lio\m=[L_{\lio,1}\m]-[L_{\lio,2}\m]$, $0\le m\le n$. Alors
$j_{!*}K_\lio = K_\lio^{(n)}$. On montre la \Ecomp de
$(K_\lio\m)_\lidI$ par récurrence sur $m$.

Le cas $m=0$ est vide. On suppose la \Ecomp de $(K_\lio\mm)_\lidI$
établie, $1\le m \le n$. $L_{\lio,\al}\m = \tYle[m]{-d_m-1}
Rj_{m*} L_{\lio,\al}\mm$, $\al=1,2$. Pour tout $x\in |U_{m-1}|$,
$L_x(K_\lio\m,t) = L_x(K_\lio\mm,t)$. \Dpr \ref{cor.pds}, pour
tout $y\in |Y_m|$, $L_y(K_\lio\m,t)$ peut être extrait de
$L_y(Rj_{m*}K_\lio\mm,t)$ comme la partie de poids $\le w-d_m-1$,
$\lidI$.
D'après \ref{th.6op}, $(Rj_{m*}K_\lio\mm)_\lidI$ est
$E$-compatible. Donc $(K_\lio\m)_\lidI$ l'est aussi.

(b) Cas général. Le résultat découle de (a) et de \ref{lm.p}, car
par \dfn,
\[f_{!*}= \bigoplus_{w\in \Z} i_{w,X} f_{!*} p_{w,U},\]
avec des
notations de \ref{lm.p}.
\end{proof}

\begin{lemm}\label{lm.p}
Soit $w$ un entier. La projection $p_{w,X}: \Km(X,-)\to K_w(X,-)$ et
l'inclusion $i_{w,X}: K_w(X,-)\to\Km(X,-)$ préservent la\/ \Ecomp.
\end{lemm}

\begin{proof}
Le résultat pour $i_w$ est trivial. On montre le résultat pour
$p_w$ simultanément pour tout $w\in\Z$. \Ops $\sharp I=1,2$.

(a) Un cas spécial. On suppose $X$ lisse sur $k$ purement de
dimemsion $d$. Soit $(K_\lio)_\lidI\in \Km(X,\Qlb)$ \Ecom avec
$K_\lio = \sum_{w\in\Z}K_{\liow}$, $K_\liow=
[L_{\lio,w,1}]-[L_{\lio,w,2}]$, où $L_{\lio,w,\al}\in\Per_w(X,\Qlb)$
et $\fH^e(L_{\lio,w,\al})$ lisse sur $X$, $\forall e\in\Z$,
$\al=1,2$, $\wiZ$, $\lidI$. Alors $L_x(K_\liow,t)$ peut être extrait
de $L_x(K_\lio,t)$ comme la partie de poids $w-d$, $\forall x\in
|X|$. Donc $\forall \wiZ$, les $p_w(K_\lio) = K_\liow$, $\lidI$,
forment un système \Ecom.

(b) Cas général. \Ops $X$ réduit. On fait une récurrence
noethérienne. Le résultat pour $p_{w,\emptyset}$ étant trivial, on
suppose le résultat pour $p_{w,Y}$ établi pour tout fermé
$Y\subsetneqq X$ (réduit).

On prend $(K_\lio)_\lidI \in \prod_\lidI \Km(X,\Qlb)$ un système
\Ecom. On pose $K_\lio = \sum_{w\in\Z}K_{\liow}$, $K_\liow=
[L_{\lio,w,1}]-[L_{\lio,w,2}]$ où $L_{\lio,w,\al}\in
\Per_w(X,\Qlb)$, $\al=1,2$. On prend un ouvert non vide $j:U\har X$
lisse sur $k$ purement de dimemsion $d$ \tq
$\fH^e(j^*L_{\lio,w,\al})$ soit lisse sur $U$, $\forall e\in\Z$,
$\al=1,2$, $\wiZ$, $\lidI$, et $i:Y\har X$ le fermé \compl. Alors
$\forall \wiZ$, $(j^*K_\liow)_\lidI$ est un système \Ecom sur $U$,
\dpr (a). Donc $(j_{!*}j^*K_\liow)_\lidI$ en est un sur $X$, \dpr
\ref{th.f!*}(a). \Dpr \ref{pr.dec1},
\beq\label{eq.decK}
K_\liow = j_{!*}j^*K_\liow + i_*K''_\liow,
\eeq
où $K''_\liow\in K_w(Y,\Qlb)$. Donc
$K_\lio=\sum_{\wiZ}j_{!*}j^*K_\liow + i_*\sum_{\wiZ}K''_\liow$,
d'où la \Ecomp de $(\sum_\wiZ K''_\liow)_\lidI$, qui implique la
\Ecomp de $(K''_\liow)_\lidI$ par l'\hyp de \rec, $\forall \wiZ$.
La \Ecomp de $(K_\liow)_\lidI$ suit alors de (\ref{eq.decK}).
\end{proof}


\section{Indépendance de $l$ et intégralité pour la cohomologie d'intersection}

Soit $X$ un schéma propre sur $k$ purement de dimension $d$. On
définit la cohomologie (\resp le complexe) d'intersection par
\[IH^i(X_\kb,\Qlb) = H^i(X_\kb,(IC_X)_\kb), \text{(\resp $IC_X = IC(X,\Qlb) = (j_{!*}(\Qlb [d]))[-d],$)}\]
où $j: U\har X$ est une immersion ouverte dominante telle que
$U_\red$ soit lisse sur $k$. Notons que la normalisation ici
diffère de celle dans \cite[0]{BBD}. D'après \ref{th.Gabber},
$IC_X$ est pur de poids $0$.

\begin{theo}\label{th.IH}
Soient\/ $X$ un schéma propre équidimensionnel sur $k=\F_q$, $l$
un
nombre premier\/ $\nmid q$. Alors pour chaque $i$, $P_i(t) = 
\det(1-t\Fr, IH^i(X_\kb,\Qlb))$ est dans\/ $\Z[t]$ et indépendant
de $l$.
\end{theo}

\begin{proof}
D'après \ref{th.f!*} et \ref{th.6op},
\[L(IC_X, t) =
\det(1-t\Fr,R\Gamma(X_\kb,(IC_X)_\kb))^{-1} \in \Q(t)\] et
indépendant de $l$. Le morphisme $a: X\to \Spec k$ est propre,
donc $Ra_*IC_X$ est pur de poids 0. Il en résulte que
$IH^i(X_\kb,\Qlb)$ est pur de poids $i$, donc les $P_i(t)$ peuvent
être extraits de $L(IC_X,t)$ de manière indépendante de $l$, et
par suite les $P_i(t)$ sont dans $\Q[t]$ et indépendants de $l$.

Il reste à démontrer l'\inte. \Ops $X$ réduit. Soit $f: X'\to X$
une normalisation. Prenons $j: U\har X$ comme plus haut. Alors
$j=fj'$ où $j': U\har X'$ est une immersion ouverte. Donc $IC_X =
f_*(IC_{X'}) = f_*(\oplus IC_{X_i})$, où $X_i$ sont les
composantes connexes de $X'$. Donc \ops $X$ intègre.

D'après \cite[4.1]{deJong}, on a une altération $\pi: Y\to X$
génériquement étale telle que $Y$ soit \irr, lisse et \proj sur
$\Spec k$. Prenons $V\subset U$ un ouvert non vide \tq $\pi_V:
Y\times_X V\to V$ soit un \rev fini étale. Alors $\Qlb$ sur $V$
est un facteur direct de $R\pi_{V*}\pi_V^* \Qlb = j_V^* K$, où
$j_V: V\har X$, $K= R\pi_* \Qlb$. Donc $\Qlb [d]$ est facteur
direct de $j_V^* (\pH^d K)$, où $d= \dim X$. Le complexe $K$ étant
pur, $\pH^d K$ l'est aussi, \dpr \ref{th.pHi}. Donc
$IC_X[d]=j_{V!*}(\Qlb [d])$ est facteur direct de $\pH^d K$,
d'après \ref{pr.dec1}. 
Il en résulte que $P_i(t)$ est facteur de $\det(1-t\Fr,
H^i(Y_\kb,\Qlb))$, \dpr \ref{cor.dec}.\footnote{Pour notre $K$
ici, $\pH^dK$ est en fait facteur direct de $K[d]$
(\cite[5.4.10]{BBD} et \cite{dec}).
}
Donc l'intégralité pour $X$ découle de celle
pour $Y$, qui est vraie d'après \cite{WeilI} ou \cite[5.2.2]{int}.
\end{proof}

On donnera une autre \dem de l'\inte au \S \ref{sec.int}. 

\section{Appendice. Théorème d'\inte}\label{sec.int}

Dans cet appendice, on fixe un nombre premier $l\nmid q$. Soit $T$
un ensemble de nombres premiers. Un \elt de $\Qlb$ est dit
$T$-\emph{entier} s'il est \alg sur $\Q$ et entier sur
$\Z[(1/t)_{t\in T}]$. Soit $X$ un schéma séparé \tf sur $k= \F_q$.
Un $\Qlb$-\faisc $\cF$ sur $X$ est dit $T$-\emph{entier} si pour
tout $x\in |X|$, les valeurs propres de l'action de $\Fr_x$ sur
$\cF_\xb$ sont $T$-entières. Des \faisx $T$-entiers sont stables
par sous-quotients et extensions. Un objet $K\in \Dbc(X,\Qlb)$ est
dit $T$-\emph{entier} si tous ses \faisx de cohomologie sont
$T$-entiers. Cette notion est stable par $\Lt, f^*, Rf_!$
\cite[5.2.2]{int}.

\begin{theo}\label{th.int}
Soient $f:X\to Y$ un morphisme de schémas séparés \tf sur\/ $\Spec
k$ et\/ $K\in \Dbc(X,\Qlb)$ $T$-entier. Alors $Rf_* K$ est\/
$T$-entier.
\end{theo}

Ce résultat est démontré dans \cite[5.6]{int} en supposant la
résolution des singularités. On peut adapter la \dem pour éliminer
l'hypothèse de résolution. En fait, le premier usage de l'\hyp de
\resol (l. 4 de la \dem \lcit[p. 396]) peut être remplacé par
\cite[Finitude]{SGA4d}. Pour le deuxième usage (bas de \cite[p.
397]{int}), 
rappelons que l'on est dans le cas suivant
\begin{multline}\label{eq.cond}
\text{$X$ normal et $K\simeq (\cG\otimes_R E)\otimes_E \Qlb$, où
$\cG$ est un $R$-faisceau lisse}\\
\text{\tq $\cG\otimes_R R/\fm$ soit constant sur chaque composante
connexe de $X$,}
\end{multline}
où $E$ est une extension finie de $\Q_l$ convenable, $R$ son
anneau des entiers, $\fm$ l'idéal maximal.

\begin{lemm}
Soit $f:X\to Y$ un morphisme de schémas séparés \tf sur\/ $S=\Spec
F$, où\/ $F$ est un corps parfait. Alors il existe un diagramme
commutatif
\[\xymatrix{X'\s \ar@{^{(}->}[r]^{j} \ar[d]_{\eps} & Z\s\ar[ddl]^{g}
& D\s \ar@{_{(}->}[l] \\
X\ar[d]_{f}\\
Y}\] où $\eps$ est un hyper-recouvrement propre et pour tout $n$,
$j_n$ est une immersion ouverte, $g_n$ propre, $Z_n$ lisse sur\/
$S$, $D_n$ un diviseur à croisements normaux dans\/ $Z_n$ de
complémentaire\/ $X_n$.
\end{lemm}

\begin{proof}
Conséquence facile de \cite[4.1]{deJong} et de \cite[6.2]{Hodge3}.
Voir \cite[2.6]{Orgogozo}.
\end{proof}

Alors
\[Rf_* K = Rf_* R\eps_*\eps^* K = Rg_* Rj_* \eps^*K,\]
donc on a une suite spectrale
\[E_1^{pq} = \fH^q(Rg_{p*} Rj_{p*} \eps_p^* K) \Rightarrow R^{p+q}f_* K.\]
$X'_p$ et $\eps_p^*K$ satisfont encore (\ref{eq.cond}), donc
$\eps_p^* K$ est lisse sur $X'_p$ et modérément ramifié le long de
$D_p$. \Dpr \cite[5.6.1]{int}, la \Tint de $\eps_p^* K$ implique
la \Tint de $Rj_{p*}\eps_p^*K$, qui donne alors la \Tint de
$Rg_{p*} Rj_{p*} \eps_p^* K$ car $g_p$ est propre. Il en suit que
$Rf_* K$ est \Tent.

\begin{coro} Soient $f:Z\to X$ un morphisme quasi-fini de schémas séparés \tf sur
$k$ et\/ $K\in\Dbc(Z,\Qlb)$ un \faisc pervers\/ $T$-entier. Alors
$f_{!*}K$ est\/ \Tent.
\end{coro}

\begin{proof}
Grâce à \ref{th.int}, \ops que $f$ soit une
immersion ouverte $j : U \har X$. 
On applique \ref{cor.j!*} pour voir que
\[j_{!*}K = \tYle[n]{-d_n -1}Rj_{n*} \cdots
\tYle[1]{-d_1 -1} Rj_{1*} K.
\]
La conclusion découle alors du \Thm~\ref{th.int}.
\end{proof}

En prenant $T=\emptyset$, on obtient que pour $a: X\to\Spec k$
vérifiant l'\hyp du \thm \ref{th.IH}, $Ra_* IC(X,\Qlb)$ est
$T$-entier, \ie les valeurs propres de $\Fr$ sur
$IH^i(X_\kb,\Qlb)$ sont entières sur $\Z$, ce qui donne une autre
\dem de l'\inte dans le \thm~\ref{th.IH}.

\end{document}